\documentclass[12pt,a4paper]{article}
\usepackage[utf8]{inputenc}
\usepackage[english]{babel}
\usepackage{amsmath}
\usepackage{etoolbox}
\usepackage{dynkin-diagrams}
\usepackage{amsfonts}
\usepackage{amssymb}
\usepackage{amsthm}
\usepackage{makeidx}
\usepackage{graphicx}
\usepackage[all]{xy}
\usepackage{tikz}

\usepackage[left=2.5cm,right=2.5cm,top=2cm,bottom=2cm]{geometry}

\newtheorem{Theorem}{Theorem}[section]
\newtheorem{Lemma}[Theorem]{Lemma}
\newtheorem{definition}[Theorem]{Definition}

\newtheorem{Proposition}[Theorem]{Proposition}

\newtheorem{Remark}[Theorem]{Remark}

\begin{document}
\title{A Cryptosystem Using Cluster Algebras}

\author{ Martin Ortiz Morales and Leticia Peña Tellez}
\date{May 11, 2026}
\maketitle

\begin{abstract}
We establish an algorithm to encrypt and decrypt messages, where messages can be seen as elements of a finite field, using of mutations in a cluster algebra finite type.
\end{abstract}

\section*{Introduction}



Cluster algebras were introduced by Fomin and Zelevinsky in 2002 as a new class of commutative algebras designed to study phenomena arising in Lie theory, especially total positivity and canonical bases \cite{fomin1}. Since their introduction, cluster algebras have developed into a broad and highly active research area connecting algebra, geometry, combinatorics, topology, representation theory, and mathematical physics.

Roughly speaking, a cluster algebra is generated from an initial set of variables, called a \emph{cluster}, together with an exchange matrix or quiver encoding combinatorial data. Through an iterative process known as \emph{mutation}, one obtains new clusters and new generators called \emph{cluster variables}. Despite the apparent complexity of repeated mutations, cluster algebras exhibit remarkable structural properties such as the Laurent phenomenon, positivity conjectures (now proved in many cases), and deep symmetries among seeds and exchange patterns.

One of the most important reasons for the success of cluster algebras is their universality. They appear naturally in the coordinate rings of algebraic varieties such as Grassmannians, flag varieties, and double Bruhat cells. In geometry, they play a central role in Teichm\"uller theory and higher Teichm\"uller spaces, where cluster coordinates provide canonical parametrizations of moduli spaces \cite{FG2006}. In representation theory, cluster categories and quiver representations provide categorifications of cluster structures, leading to powerful methods for understanding canonical bases and derived categories.

Recent years have seen substantial progress in the interaction between cluster algebras and mathematical physics. In scattering amplitudes, cluster structures govern symbol alphabets and adjacency relations in planar $\mathcal{N}=4$ supersymmetric Yang--Mills theory, revealing hidden combinatorial constraints in Feynman integrals and amplitude calculations \cite{Chicherin2020}. Likewise, cluster methods have become relevant in integrable systems, where mutations correspond to discrete dynamical transformations and periodicity phenomena \cite{GekhtmanIzosimov2024}.

Another modern direction concerns Poisson geometry and quantum cluster algebras. Recent work has studied symplectic leaves, Azumaya loci, and root-of-unity quantizations of upper cluster algebras, showing that cluster structures remain meaningful far beyond the commutative setting \cite{BrownYakimov2024}. These developments suggest that cluster algebras provide a unifying framework for classical and noncommutative geometry.

Cluster algebras have also found applications in topology and low-dimensional geometry. Connections with knot invariants, braid varieties, and symplectic topology continue to expand, with recent surveys highlighting interactions between cluster mutations and geometric structures on surfaces and braid moduli spaces \cite{Oberwolfach2024}.

Because of their rich combinatorial definition and unexpectedly wide range of applications, cluster algebras are now regarded as one of the central unifying theories in contemporary algebra. Their ongoing development continues to generate new tools and perspectives across mathematics and theoretical physics.

\def\row#1/#2!{#1_{\IfStrEq{#2}{}{n}{#2}} & \dynkin{#1}{#2}\\}
\newcommand{\tble}[1]{
  \renewcommand*\do[1]{\row##1!}
  \[
  \begin{array}{ll}\docsvlist{#1}\end{array}
  \]
}

In this work we establish an algorithm to encrypt and decrypt messages, where messages can be seen as elements of a finite field, by using  mutations in a cluster algebra of finite type. Suppose that two entities $A$ and $B$ want exchange information of sure way, we want to ensure that if any entity other than A and B intercepts the information for the purpose of knowing its content, this will be a difficult or impossible task for him.The goal of a cryptosystem is will do that a message $m$ is a message illegible for its transmission, let say  $c$ (encrypt) that is the message which sending $A$ to $B$. When $B$ receives $c$, $B$ has to know as recuperate $m$ (decrypt). We briefly describe here what the process is, which will be shown in detail later.

We will see that the message $m$ can be seeing as combination lineal of cluster variables and as part of a cluster as follows. Suppose that $m$ is a letter belonging to an alphabet of size $N$, 
then a  number prime $p$ is chosen and a suitable number $r\in \mathbb{N}$ so $q=p^r$ is close to $N$,  thus  $m$ can be seen as an element other than zero in a finite field $\mathbb{F}_{q}$.  Note that  $\mathbb{F}_{q}=\mathbb{Z}_p[x]/\langle f(x) \rangle$, where $q=p^r,$  and $r$ is  degree of an irreducible polynomial $f(x)$ in $\mathbb{Z}_p$, and $\mathbb{F}_{q}$ is an extension field of $\mathbb{Z}_p$ of degree $r$. Let $\alpha \in \mathbb{F}_q$ a root of $f(x)$ in $\mathbb{F}_{q}$, if $\beta \in \mathbb{F}_q$, $\beta$  can  be represented as $$\beta = a_0 + a_1 \alpha + \cdots + a_{r-1}\alpha^{r-1}, a_i \in \mathbb{Z}_p.$$

Also $\lbrace 1,\alpha,\ldots, \alpha^{r-1} \rbrace$ is a base for $\mathbb{F}_q$, seen as a vector space  over  $\mathbb{Z}_p$. The isomorphism
\begin{equation}
\begin{array}{rlcl}
\varphi: & \mathbb{F}_q & \rightarrow & \mathbb{Z}^r_p \\
    & \beta & \rightarrow & (a_0,a_1,\ldots, a_{r-1})
\end{array},
\label{eq3}
\end{equation} give a representation of the elements of finite field $\mathbb{F}_q$, which we use in the cryptosystem  that is  explained in this work.

On the other hand, we denote by $\Gamma_r$ the quiver corresponding to a Dynkin diagram $\Gamma$ with $r$ vertices,  and by  $\mathcal{A}_{r}$  the cluster algebra associated to the Dynkin diagram corresponding to the quiver $\Gamma_r$.

For the cluster algebra  $\mathcal{A}_r$ we consider as initial seed the cluster $\tilde{x}=\lbrace x_0,x_1,\ldots,x_{r-1} \rbrace$ and as exchange matrix $\tilde{B}$ the adjacency matrix that represent  the quiver $\Gamma_r$ which is obtained from the Dynkin diagram $\Gamma,$  where the vertices of $\Gamma_r$ are labelled by $x_0,x_1,\ldots,x_{r-1}$. Now, we can associate each cluster variable  $x_i$ of the initial cluster $\tilde{x}$ to the element $\alpha^i \in \mathbb{F}_q$ by the bijection $\alpha_i\mapsto x_i$,  according  Proposition \ref{bijection}.
Thus,  $m=(a_0,a_1,\ldots,a_{r-1}) \in \mathbb{Z}_p^r$, also seen as element in $\mathbb{F}_q$, $m= a_0\alpha^0+a_1\alpha^1 + \cdots + a_{r-1}\alpha^{r-1}$, can be seen as a linear combination of cluster variables $m= a_0x_0+a_1x_1 + \cdots + a_{r-1}x_{r-1}$.

Finally, to encrypt the message $m$, we use a key encryption, $k$, which is a key private that is only its know the identities $A$ and $B$, and the same key help us to decipher the message.

The key encryption
\begin{equation}
k=\{k_0,k_1,\ldots,k_t \},
\label{eq6}
\end{equation}
is an integers sequence such that $0 \leq k_i \leq r-1$. The integers $k_1,\ldots, k_{t}$ represent a sequence of mutations in the cluster algebra $\mathcal{A}_r$.  The idea that we show is that we  can "keep" the message $m$ in a quiver $\Gamma_r'$ and "hide" it using mutations. The quiver $\Gamma'_r$ is constructed the same form that $\Gamma_r$, except in the labelled of the vertex $x_{k_0}$, we write $m=a_0x_0+a_1x_1 + \cdots + a_{r-1}x_{r-1}$ instead  $x_0$.

The encrypt message $\tilde{c}$, is obtained as part of the seed:
$$(\tilde{c},\tilde{B'})=\mu_{k_{t}}(\mu_{k_{t-1}}(\cdots \mu_{k_2}(\mu_{k_1}(\tilde{x},\tilde{B})))),$$ and subsequently replacing each cluster variable $x_{k_0}$ by $m= a_0x_0+a_1x_1 + \cdots + a_{r-1}x_{r-1}$. Thus, the seed $(\tilde{c},\tilde{B'})$ is send to $B$, where $\mu_{k_i}$, $1\le i\le t$ are mutations.

This work consists of four sections. In Section 1, we provide the necessary preliminaries to familiarize the reader with the terminology used throughout. We offer a brief introduction to cluster algebras, covering root systems and Dynkin diagrams. In Section 2, we explain how the proposed cryptsystem works. Subsequently, in Section 3, we provide two concrete examples of how a message can be encrypted and decrypted using the proposed method, and how the process can be visualized with the help of the open-source computer
algebra system \verb"sage" \cite{Sage}. Finally, in Section 4, we discuss the efficiency of the proposed method.

\section{Preliminaries}

\subsection{Root systems}

Following \cite{algebralie}, we present some important results  that we  use  throughout this in this paper.

Let $E$ be a finite dimensional real vector space endowed with an inner product written $(-,-)$. Given a non zero vector $v \in E$, let $s_v$ be the reflection 
in the hyperplane normal to $v$, $H_v$.

\begin{figure}[h!]
\centering
\begin{tikzpicture}[scale=2]
  \draw[fill=gray!10, thick]
    (-1,-0.3) -- (1,-0.3) -- (1.3,0.3) -- (-0.7,0.3) -- cycle;
  \node at (0,0.05) {$\ \ \ \ \ \ \ H_v$};

  \draw[->,thick] (0,0.05) -- (0,1) node[above] {$v$};

  \draw[dashed,->,thick] (0,0.05) -- (0,-0.9) node[below] {$S_v(v)$};
\end{tikzpicture}
\caption{The reflexión $s_v$ in the  hyperplane  \(H_v\).}
\label{fig:reflexion}
\end{figure}
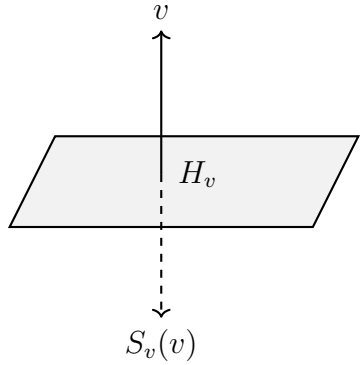

The reflection $s_v$ has the following properties:
\begin{itemize}
\item [(i)] $s_v(v)=-v$
\item  [(ii)] $s_v$ fixes all elements $y$ such that $y \in H_v$.
\item [(iii)] For all $x \in E$, the reflection $s_v$ of $x$ can calculate as
\begin{equation}
s_v(x)=x-\frac{2(x,v)}{(v,v)}.
\label{eq31}
\end{equation}
\item [(iv)] $s_v$ preserves the inner product, that is,
$$(s_v(x),s_v(y))=(x,y) \text{ para toda } x,y \in E.$$
\item [(v)] As it is a very useful convention, we shall write \begin{equation}
\langle x,v \rangle:=\frac{2(x,v)}{(v,v)},
\label{eq32}
\end{equation}
the symbol $\langle x,v \rangle$ is only lineal with respect to its first variable.
\end{itemize}

\begin{definition}\label{rootsystem}
A subset $R$ of a real inner product space $E$ is a \textit{root system} if it satisfies the following axioms.
\begin{description}
\item[(R1)] $R$ is finite, it spans $E$, and it does not contain 0.
\item[(R2)] If $\alpha \in R$, then only scalar multiples of $\alpha$ in $R$ are $\pm \alpha$.
\item[(R3)] If $\alpha \in R$, then the reflection $s_\alpha$ permutes the elements of $R$.
\item[(R4)] If $\alpha, \beta \in R$, then $\langle \beta, \alpha \rangle \in \mathbb{Z}$.
\end{description}

The elements of $R$ are called \textit{roots}.
\end{definition}

\begin{definition}
A subset $B$ of $R$ is a \textit{base} for the root system $R$, if
\begin{description}
\item [(B1)] $B$ is a vector space basis for $E$, and
\item [(B2)] every $\beta \in R$ can be written as $\beta = \sum_{\alpha \in B} k_{\alpha}\alpha$, with $k_{\alpha} \in \mathbb{Z}$, where all the non zero coefficients $k_{\alpha}$ have the same sign.
\end{description}
\end{definition}

In \cite{algebralie} it is proved that all root systems have a base. Since for each $\alpha \in R$, the reflection $s_{\alpha}$ which acts as an invertible linear map on $E$,  we may therefore consider the group of invertible linear transformations of $E$ generated by the reflections $s_{\alpha}$ for $\alpha \in R$. This is known as the \textit{Weyl group} of $R$ denoted by $W$ or $W(R)$. 

By  (R1) in  Definition \ref{rootsystem}, $R$ is finite and by (R3) the Weyl group $W$, associated with $R$ is finite. In \cite{algebralie}, it is proved that if $W_0:= \langle s_{\gamma}: \gamma \in B \rangle$, where $B$ is a base for a root system $R$, then for the Weyl group associated with $R$, $W$, is such that $W=W_0$.

We consider $B$ a base for root system $R$ and we stablish an order on elements of $B$, say $\lbrace \alpha_1 , \ldots , \alpha_l\rbrace$.

\begin{definition}
The \textit{Cartan matrix} of $R$ is defined to be the $l \times l$ matrix with $ij$-th entry $\langle \alpha_i , \alpha_j \rangle$.
\end{definition}

The Cartan matrix depends only on the ordering adopted with our chosen base $B$ and not on the base itself.

\begin{definition}
If $B=\left( b_{ij}\right)$ is a square matrix $n \times n$ with integer entries, its counterpart of Cartan is a square matrix $n \times n$, $A=(a_{ij})$ where
\begin{equation*}
a_{ij}=\left\{
\begin{array}{ll}
2 & \text{if } i=j,\\
-\vert b_{ij}\vert & \text{if } i\neq j.
\end{array} \right.
\end{equation*}
\end{definition}

In \cite{algebralie} is proved the following lemma. This lemma give all the possibilities for $\langle \alpha, \beta \rangle \langle \beta, \alpha \rangle$, which depend of the angle between the vectors $\alpha$ and $\beta$.

\begin{Lemma}[Finiteness Lemma]
Suppose that $R$ is a root system in the real inner product space $E$. Let $\alpha, \beta \in R$ with $\beta \neq \pm \alpha$, then
$$\langle \alpha, \beta \rangle \langle \beta, \alpha \rangle \in \lbrace 0,1,2,3 \rbrace.$$
\label{lem11}
\end{Lemma}

Another way to record the information given in the Cartan matrix is in a graph $\Gamma = \Gamma(R)$, defined as follows. The vertices of $\Gamma$ are labelled by the simple roots of $B$. Between the vertices labelled by simple roots $\alpha$ and $\beta$, we drawn $d_{\alpha\beta}$ lines, where
$$d_{\alpha\beta}:=\langle \alpha , \beta \rangle \langle \beta , \alpha \rangle \in \lbrace 0, 1, 2, 3 \rbrace.$$

If $d_{\alpha\beta}>1$, we draw an arrow pointing from the longer root to the shorter root. This graph is called the \textit{Dynkin diagram} of $R$. In \cite{algebralie}, is proved that the Dynkin diagram of $R$ is independent of the choice of base.

\begin{center}
\begin{figure}[hbtp]
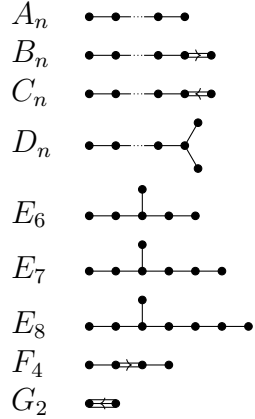

\centering
\tble{A/{},B/{},C/{},D/{},E/6,E/7,E/8,F/4,G/2}
\caption{Dynkin diagrams}
\end{figure}

\end{center}


\subsection{Cluster algebras}

Cluster algebras were introduced by Sergey Fomin and Andrei Zelevinsky in 2002 (cf. \cite{fomin1}). Before giving a definition of cluster algebras, some definitions and previous results are necessary.

Fix $n,m \in \mathbb{N}$ such that $n\leq m$ and usually $I=\lbrace 1,2,\ldots,n \rbrace$. Let $\mathbb{F}=\mathbb{Q}(u_1,\ldots,u_m)$ be the field of rational functions in the indeterminates $u_1,\ldots, u_m$ with coefficients in $\mathbb{Q}$. Thus the elements of $\mathbb{F}$ are of the form $\dfrac{f(u_1,\ldots,u_m)}{g(u_1,\ldots,u_m)}$, where $f,g$ are polynomials with coefficients in $\mathbb{Z}$. A cluster algebra (of geometric type) is a subring of $\mathbb{F}$ determined by combinatorial data. The data determine a subset of $\mathbb{F}$ which generates the cluster algebra.

\begin{definition}
A \textbf{seed} is a pair $(\tilde{x}, \tilde{B})$ where

\begin{description}
\item[a) ] $\tilde{x} = \lbrace x_1, \ldots, x_m \rbrace$ is a free generating set of $\mathbb{F}$ over $\mathbb{Q}$; that is, $\mathbb{F}$ is generated as a field over $\mathbb{Q}$ by $\tilde{x}$, and the $x_i$ for $1\leq i \leq m$ are algebraically independent.

\item[b) ] $\tilde{B}=(b_{ij})$ is an $m\times n$ integer matrix where the first $n$ rows and the columns correspond to the elements of $\mathbf{x}= \lbrace x_1,\ldots,x_n \rbrace$ and the last $m-n$ rows correspond to the elements of $\mathbf{c}=\lbrace x_{n+1}, \ldots, x_m \rbrace$.

\item[c) ] The \textbf{principal part} $\textbf{B}=(b_{ij})_{1\leq i,j \leq n}$ of $\tilde{B}$ is sign skew symmetric matrix, that is for all $1 \leq i,j \leq n$, either $b_{ij}=b_{ji} = 0$, or $b_{ij}$ and $b_{ji}$ are non zero and of opposite sign. In particular, $b_{ii}$ must be zero for all $i$.
\end{description}
\end{definition}

The set $\textbf{x}$ is known as the \textit{cluster} of the seed, with its elements known as \textit{cluster variables}. The elements of the set $\textbf{c}$ are known as \textit{coefficients}, \textit{stable variables} or \textit{frozen variables}, and the union $\tilde{x}=\textbf{x} \cup \textbf{c}$ is know as the \textit{extended cluster} of the seed. The matrix $\tilde{B}$ is known as an \textit{exchange matrix}.

\begin{definition}
Fix $k \in I$. We define the \textbf{mutation} $\mu_k(\tilde{x},\tilde{B})=(\tilde{x'},\tilde{B'})$ of $(\tilde{x},\tilde{B})$ in the direction $k$ as follows: the cluster $\tilde{x'}$ is equal to $\tilde{x}$ with $x_k$ replaced by a new elements of $\mathbb{F}$, denoted by $x'_k$, satisfying the exchange relation 
\begin{equation}
x_kx'_k = \prod^{m}_{\substack{j=1 \\ b_{kj}>0}}x_{j}^{b_{kj}}+ \prod^{m}_{\substack{j=1 \\ b_{kj}<0}}x_{j}^{-b_{kj}}.
\label{eq1}
\end{equation}

Note that, as $b_{kk}=0$, $x_k$ does not appear on the right hand side of (\ref{eq1}). Furthermore, $\tilde{B}$ will be replaced by $\tilde{B'}=(b'_{ij})_{\substack{1 \leq i \leq m,\\ 1 \leq j \leq n}}$, where
\begin{equation}
b'_{ij}=\left\{
\begin{array}{ll}
-b_{ij} & \text{if } i=k \text{ or } j=k,\\
b_{ij}+\dfrac{|b_{ik}|b_{kj} + b_{ik}|b_{kj}|}{2} & \text{otherwise}.
\end{array} \right.
\end{equation}
\label{def6}
\end{definition} 

\begin{definition}
Two seeds $(\tilde{x},\tilde{B})$ and $(\tilde{y},\tilde{C})$ are said to be \textbf{equivalent} if there is a permutation $\pi$ of $\{ 1,\ldots, n, n+1, \ldots, m\}$ such that
\begin{description}
\item[a)] $\pi(I)=I,$
\item[b)] $\pi(i)=i$ for all $i \in \{ n+1, \ldots, m\}$,
\item[c)] $y_i=c_{\pi(i)}$ for all $i$,
\item[d)] $\tilde{C}_{i,j} = \tilde{B}_{\pi(i),\pi(j)}$ for all $i,j$.
\end{description}
\end{definition}

In \cite{fomin1} is stablished that there is a involution over the set of seeds $\mathcal{A}$, in other notes can write as the following lemma. 

\begin{Lemma}
Let $(\tilde{x},\tilde{B})$ be a seed. Then $\mu^2_k(\tilde{x},\tilde{B})=(\tilde{x},\tilde{B})$.
\label{lema1}
\end{Lemma}

\begin{proof}
Let $(\tilde{x'},\tilde{B'})=\mu_{k}(\tilde{x},\tilde{B})$ and $(\tilde{x''},\tilde{B''})=\mu_{k}(\tilde{x'},\tilde{B'})$. Note that $b'_{kj}=-b_{kj}$ then

$$\begin{array}{rl}
{x'}_k{x''}_k & = \prod^{m}_{\substack{j=1 \\ {b'}_{kj}>0}}{x}_{j}^{{b'}_{kj}}+ \prod^{m}_{\substack{j=1 \\ {b'}_{kj}<0}}{x}_{j}^{{-b'}_{kj}} \\
 & = \prod^{m}_{\substack{j=1 \\ {b}_{kj}<0}}{x}_{j}^{-{b}_{kj}}+ \prod^{m}_{\substack{j=1 \\ {b}_{kj}>0}}{x}_{j}^{{b}_{kj}} \\
 & = x_kx'_k.
\end{array}$$

To show that $\tilde{B''}=\tilde{B}$, we proof ${b''}_{ij}=b_{ij}$ for all $1\leq i,j \leq m$.

\begin{itemize}
\item If $i=k$ or $j=k$, ${b''}_{ij}=-{b'}_{ij}=b_{ij}$.
\item If $i\neq k$ and $j \neq k$ so
$$\begin{array}{rl}
{b''}_{ij} & = {b'}_{ij}+\dfrac{|{b'}_{ik}|{b'}_{kj}+{b'}_{ik}|{b'}_{kj}|}{2} \\
 & = b_{ij} + \dfrac{|b_{ik}|b_{kj}+b_{ik}|b_{kj}|}{2} + \dfrac{|-b_{ik}|(-b_{kj})+(-b_{ik})|-b_{kj}|}{2} \\
  & = b_{ij} + \dfrac{|b_{ik}|b_{kj}+b_{ik}|b_{kj}|}{2} - \dfrac{|b_{ik}|b_{kj}+b_{ik}|b_{kj}|}{2} \\
  & = b_{ij}.
\end{array}$$
\end{itemize}

\end{proof}

\begin{definition}
Two seeds (or exchange matrix) are said to be \textbf{mutation equivalent} if there is a sequence of mutation taking one to the other. To the corresponding equivalence class, it is called \textbf{mutation class}.
\label{def1}
\end{definition}

Given a seed $(\tilde{x},\tilde{B})$, let $\mathcal{S}$ be the set of seeds obtained from $(\tilde{x},\tilde{B})$ by arbitrary finite sequences of mutations.

\begin{definition}[Cluster Algebra]
Suppose that the principal parts of all the matrices in the seeds in $\mathcal{S}$ are sign skew symmetric. The \textbf{cluster algebra} with initial seed $(\tilde{x},\tilde{B})$, denoted $\mathcal{A}(\tilde{x},\tilde{B})$, is the subring of $\mathbb{F}$ generated by $x_{n+1},\ldots, x_m$ and the elements of all the clusters of the seeds in $\mathcal{S}$. The \textit{rank} of the cluster algebra is $n$. If $K$ is any field, the cluster algebra over $K$ is defined in the same way, replacing $\mathbb{F}$ with $K(u_1,\ldots,u_m)$ and taking the $K-$subalgebra generated by the cluster variables and coefficients.
\end{definition}

For this work $K$ is the finite field $\mathbb{F}_q$, with $q=p^r$ where $p$ is a prime number and $r\in\mathbb{N}$; also suppose that $n=m$.

\begin{definition}[Normalized cluster algebra]
Let $\mathcal{S}$ be a set of seeds in $\mathcal{A}(\tilde{x},\tilde{B})$ with the following properties:
\begin{enumerate}
\item Every seed $(\tilde{x'},\tilde{B'}) \in \mathcal{S}$ admits mutations in all $n$ conceivable directions, and the results of all these mutations belong to $\mathcal{S}$.
\item Any two seeds in $\mathcal{S}$ are obtained from each other by a sequence of mutations.
\end{enumerate}
\end{definition}

In \cite{fomin1}, an exchange graph of cluster algebra is defined as

\begin{definition}\label{exchgraph}

The \textbf{exchange graph} of cluster algebra, $\mathcal{A}(\tilde{x},\tilde{B})$, is a graph $n$-regular whose vertices are labelled by the seeds in $\mathcal{S}$, joined by an edge if there is a mutation from each other.
\label{def11}
\end{definition}

It is often quite convenient to use a diagram notation for exchange matrices. Let $B$ be a skew symmetric integer $n \times n$ matrix. Then $B$ can be recorded as a directed graph, better known as a \textbf{quiver} $Q(B)$, with vertices corresponding to its rows and columns, the elements of $I$. If $b_{ij}>0$, there are $b_{ij}$ arrows from $i$ to $j$. Thus, if $b_{ij}<0$, there will be $-b_{ij}=b_{ji}$ arrows from $j$ to $i$. Note that $Q(B)$ has no loops(1-cycles) or 2-cycles. If the exchange matrix is not square, there are $m$ vertices. The vertices $i$ such that $i \in \lbrace n+1,\ldots,m \rbrace$, corresponding to the coefficients, are often drawn as empty circles, and are know as \textit{frozen vertices}, the arrows are drawn the same way. From this definition of quiver, there is a bijection between $n \times n$ skew symmetric integer matrices and quivers without loops or 2-cycles.

From similar way to the definition mutation of seed in $k$, is defined \textbf{quiver mutation} at vertex $k \in I$.  

\begin{enumerate}
\item For all paths of the form $i \rightarrow k \rightarrow j$ add an arrow from $i$ to $j$. The multiplicity is taken into account, that is if there are $a$ arrows from $i$ to $k$ and $b$ arrows from $j$ to $k$, we add $ab$ arrows from $i$ to $j$.

\item Cancel a maximal set of 2-cycles from those created in (1).
\item Reverse all arrows incident with $k$.
\end{enumerate}

\subsection{Cluster algebras of finite type}

A cluster algebra is said to be of \textit{finite type} if it has finitely many seeds. It is said to be of \textit{finite mutation type} if the set of principal parts of its exchange matrices is finite.  

We consider classification the cluster algebras of finite type by S. Fomin and A. Zelevinsky in 2003 (cf. \cite{fominfinito}). This classification is give in terms of finite type Cartan matrices(or Dynkin diagrams).

In particular, the quivers that we use for this work are those associated with Dynkin diagrams (cf.\cite{algebralie}), which generate cluster algebras of finite type. We propose to use Dynkin diagrams of finite type $A_n$, $B_n$, $C_n$ and $D_n$.

Let $B$ an $n \times n$ sign skew symmetric integer matrix whose mutation class contains only sign skew symmetric matrices. Then, any $m \times n$ integer matrix $\tilde{B}$ with principal part $B$ defines a cluster algebra. 

\begin{definition}
Let $\mathcal{A}(\tilde{x},\tilde{B})\subseteq \mathbb{F}$ and $\mathcal{A}(\tilde{y},\tilde{C})\subseteq \mathbb{F}'$ be cluster algebras. We say that $\mathcal{A}(\tilde{x},\tilde{B})$ and $\mathcal{A}(\tilde{y},\tilde{C})$ are strongly isomorphic (or isomorphic as cluster algebras) if there is a field isomorphism $\sigma : \mathbb{F}\rightarrow\mathbb{F}'$ such that $(\sigma(\tilde{x}),\tilde{B})$ is a seed in $\mathcal{A}(\tilde{y},\tilde{C})$. 
\label{def12}
\end{definition}
We denote $\mathcal{A}(\tilde{B})$ to the cluster algebra whose exchange matrix is the matrix $\tilde{B}$; that is, it is a cluster algebra defined by $B$.

\begin{definition}
The collection of all cluster algebras defined by $B$ is called family $Cl(B)$. Two families $Cl(B)$ y $Cl(B')$ are said to be strongly isomorphic provided each cluster algebra $\mathcal{A}(\tilde{B}) \in Cl(B)$ is strongly isomorphic to a cluster algebra $\mathcal{A}(\tilde{B'}) \in Cl(B')$, and vice versa.  
\end{definition}

For two matrices $\tilde{B}$ and $\tilde{B'}$, we write $\tilde{B} \sim_{mut} \tilde{B'}$ if $\tilde{B}$ and $\tilde{B'}$  lie in the same mutation class (as Definition \ref{def1}).

In \cite{fominfinito}, S. Fomin and A. Zelevinsky prove the following results.

\begin{Lemma}
Two families $Cl(B)$ and $Cl(B')$ are strongly isomorphic if and only if $B\sim_{mut} B'$.
\end{Lemma}

\begin{Theorem}

\begin{description}
\item[a)] All cluster algebras in a family $Cl(B)$ are simultaneously of finite or infinite type.
\item[b)] A cluster algebra is of finite type if and only if the Cartan counterpart of the principal part of one of its seeds is a Cartan matrix of finite type.
\item[c)] The families of cluster algebras of finite type are classified up to strong isomorphism by the Cartan matrices of finite type(up to simultaneous permutation of rows and columns), with the Cartan matrix associated to a given family of cluster algebras arising as in (b) from any cluster algebra in the family.
\end{description}
\label{teo1}
\end{Theorem}

In \cite{grafvaluada}, its defined a \textbf{valued graph} $G$ on vertices $1,\ldots,n$ as a graph with no loops and at most one
arrow between any pair of vertices, together with integers $v_{ij}$ for all $i,j$ for which there is an edge between $i$ and $j$, such that there are positive integers $d_1,\ldots,d_n$ such that $d_iv_{ij} = d_jv_{ji}$ for all $i,j$, that is the matrix $(v_{ij})$(with zero entries where there is no edge) is symmetrizable. If $v_{ij} = v_{ji} = 1$ for all $i, j$, then $G$ is regarded as a graph in the usual sense. We write $v_{ij}$ on the edge between $i$ and $j$ near $i$ and $v_{ji}$ near $j$. For edges with $v_{ij} = v_{ji} = 1$, we omit the labels. An orientation of a valued graph is just an orientation of its underlying graph, and is known as a \textbf{valued quiver}.

\begin{Theorem}
\begin{description}
\item[a)] All cluster algebras in a family $Cl(B)$ are simultaneously of finite or infinite type.
\item[b)] A cluster algebra is of finite type if and only if the underlying valued graph of the principal part of one of its valued quiver seeds is a Dynkin diagram (with the above identification). 
\item[c)] The families of cluster algebras of finite type are classified up to strong isomorphism by the Dynkin diagrams, with the Dynkin diagram associated to a given family of cluster algebras arising as in (b) from any cluster algebra in the family.
\end{description}
\label{teo2}
\end{Theorem}

We write $\mathcal{A}_{\Gamma}$ for the cluster algebra of finite type corresponding to Dynkin diagram $\Gamma$. $\Delta$  is a simple root system of $R$, that is a base which elements have positive coefficients. We consider $R^{+}\subset R $ as the set of roots such that are positive lineal combination of elements in $\Delta$, and $R^{-}\subset R $ the set of roots such that are negative lineal combination of elements in $\Delta$. With this notation $R_{\geq -1}=R^{+}\cup(-\Delta)$.

The following result show the relation between cluster algebra of finite type and the root system corresponding the same Dynkin diagram $\Gamma$ (cf. \cite{teor1},\cite{teor2}).

\begin{Theorem}\label{T16}
Let $\mathcal{A}_{\Gamma}$ be a cluster algebra of finite type. Considerer $\tilde{x}= \{ x_1, \ldots, x_n \}$ an initial cluster for $\mathcal{A}_{\Gamma}$ and $R$ a root system corresponding to $\Gamma$ with $\Delta = \{ \alpha_1, \ldots , \alpha_n \}$. For a root $\alpha = d_1\alpha_1 + \cdots + d_n \alpha_n \in R$, let $x^{\alpha}=x_1^{d_1}\cdot \cdots x_n^{d_n}$. Then there is a bijection $\alpha\mapsto x[\alpha]$ between $R_{\geq -1}$ and the cluster variables of $\mathcal{A}$, such that $$x[\alpha]= \frac{P_{\alpha}(x_1, \ldots, x_n)}{x^{\alpha}},$$ where $P_{\alpha}$ is a polynomial in $x_1,\ldots,x_n$ with positive integer coefficient, and non zero constant term. In particular $x[-\alpha_i]=x_i$ for $i=1,\ldots,n$.
\end{Theorem}

Note that if $\alpha=d_1\alpha_1 + \cdots + d_n \alpha_n$ is a root of root system $R$. Then by definition of base $d_i \in \mathbb{Z}$, $1\leq i\leq n$,  and all coefficient other than zero $d_i$ have the same sign. Thus the elements other than zero in $\lbrace d_1, \ldots, d_n \rbrace$ are all positive or all negative.  To have a bijection as established by the Theorem \ref{T16} is necessary $\alpha \in R_{\geq -1}$, so the only ones negative roots are  precisely $\alpha = -\alpha_i$, $1\leq i \leq n$,  and again by Theorem \ref{T16} we have  $x[-\alpha_i]=x_i$. Other roots in $R_{\geq -1}$ are lineal combination of elements of $\Delta$ and its coefficients are $d_i \in \mathbb{Z}_{\geq0}$ for all $1\leq i \leq n$. Thus $x_1^{d_1}\cdot \cdots x_n^{d_n}$ is a monomial that is product of cluster variables with positive exponents.

\section{Cryptosystem}

Since we will be using finite fields, firstly it is necessary to remember their algebraic structure and characterization.

\begin{Remark}

Consider the finite field $\mathbb{F}_{q}=\mathbb{Z}_p[x]/\langle f(x) \rangle$, where $q=p^r$, $p$ is a number prime and $r$ is the degree of an irreducible polynomial $f(x)$ in $\mathbb{Z}_p$. $\mathbb{F}_{q}$ is an extension field of $\mathbb{Z}_p$ of degree $r$.

Let $\alpha \in \mathbb{F}_q$ a root of $f(x)$ in $\mathbb{F}_{q}$, if $\beta \in \mathbb{F}_q$, $\beta$ will can be represented as $$\beta = a_0 + a_1 \alpha + \cdots + a_{r-1}\alpha^{r-1}, a_i \in \mathbb{Z}_p.$$

Also $\lbrace 1,\alpha,\ldots, \alpha^{r-1} \rbrace$ is a base for $\mathbb{F}_q$, seen as space vector of $\mathbb{Z}_p$. The isomorphism
\begin{equation}
\begin{array}{rlcl}
\varphi: & \mathbb{F}_q & \rightarrow & \mathbb{Z}^r_p \\
    & \beta & \rightarrow & (a_0,a_1,\ldots, a_{r-1})
\end{array},
\label{eq3}
\end{equation} gives a representation of the elements of finite field $\mathbb{F}_q$, which we use in the cryptosystem  that is  explained in this work.
\end{Remark}

Suppose two entities, $A$ and $B$, wish to exchange information securely. We want to ensure that if any entity outside of $A$ and $B$ intercepts the information, it will be difficult or impossible for them to decrypt it.

The goal of a cryptographic system is to make a message, denoted $c$ (encrypted), illegible during transmission. This message is the one $ A $ sends to $B$. 
When $ B$ receives $c$, it must know how to recover (decrypt) it.

In this section, the message $m$ can be seeing as a combination lineal of cluster variables and as part of a cluster

Let $m$ the message that  $A$ want to send  to $B$. We choose a number prime $p$ and a number $r\in \mathbb{N}$ so $q=p^r$. The message $m$ can be seen as an element other than zero in $\mathbb{F}_{q}$ as in (\ref{eq3}).  

We denote by $\Gamma_r$ the quiver corresponding to the Dynkin diagram $\Gamma$ with $r$ vertices, where $r$ is the degree of the irreducible polynomial $f(x)$ in $\mathbb{Z}_p$. Remember that we propose to use the Dynkin diagrams of type $A_r$, $B_r$, $C_r$ and $D_r$. Then we denote by $\mathcal{A}_{r}$  the cluster algebra associated to the Dynkin diagram corresponding to the quiver $\Gamma_r$.

For the cluster algebra  $\mathcal{A}_r$ we consider as initial seed the cluster $\tilde{x}=\lbrace x_0,x_1,\ldots,x_{r-1} \rbrace$ and as exchange matrix $\tilde{B}$ the adjacency matrix that represents the quiver $\Gamma_r$ which is obtained from the  Dynkin diagram $\Gamma$, where the vertices of $\Gamma_r$ are labelled by $x_0,x_1,\ldots,x_{r-1}$. We can associated for each cluster variable of initial cluster $x_i$ the element $\alpha^i \in \mathbb{F}_q$ according to the next result.

\begin{Proposition}\label{bijection}
Let  $A=\lbrace x_0,\ldots, x_{r-1} \rbrace$ be the set of cluster variables of initial cluster for $\mathcal{A}_r$, and let  $B=\{1,\alpha,\ldots,\alpha^{r-1}\}\subset \mathbb{F}_q$ be the set such that $\alpha$ is a root of $f(x)$ in $\mathbb{F}_q$. Thus the function $\sigma: A \rightarrow B$ such that $\sigma(x_i)=\alpha^i$ is a bijection.
\label{prop1} 
\end{Proposition}

\begin{proof}
The cardinality of sets is finite and the same. By definition of $\sigma$  we have  that $\sigma$ is bijective.
\end{proof}

Note that from (\ref{eq3}), $m=(a_0,a_1,\ldots,a_{r-1}) \in \mathbb{Z}_p^r$, also can be seen as element in $\mathbb{F}_q$, $m= a_0\alpha^0+a_1\alpha^1 + \cdots + a_{r-1}\alpha^{r-1}$. Therefore, by Proposition \ref{prop1}, $m= a_0x_0+a_1x_1 + \cdots + a_{r-1}x_{r-1}$; that is $m$ can be seen as a a linear combination of cluster variables. 





\subsection{Encryption of the message}

For encryption of the message, we use a key encryption, $k$, which is a key private which is only known by
 the identities $A$ and $B$, and the same key help us to decipher the message.

The key encryption
\begin{equation}
k=\{k_0,k_1,\ldots,k_t \},
\label{eq6}
\end{equation}
is an integers sequence such that $0 \leq k_i \leq r-1$. The integers $k_1,\ldots, k_{t}$ represent a sequence of mutations in the cluster algebra $\mathcal{A}_r$.

The first term in the sequence, $k_0$, indicates the position where is saving the message $m$. A quiver $\Gamma'_r$ is constructed the same form that $\Gamma_r$, except in the labelled of the vertex $x_{k_0}$, instead we write $m=a_0x_0+a_1x_1 + \cdots + a_{r-1}x_{r-1}$. As $\Gamma_r$ and $\Gamma'_r$ have associate the same Dynkin diagram $\Gamma$, by Theorems \ref{teo1} and \ref{teo2}, the cluster algebra associated to $\Gamma_r$ and $\Gamma'_r$ are isomorphic strongly. By example, for the quivers $A_3$ and $A'_3$ with $k_0=1$.

\begin{center}
\begin{tabular}{cc}
\xymatrix{
x_0 \ar[r] & x_1  & x_2 \ar[l] }

 & 

\xymatrix{
x_0 \ar[r] & m  & x_2 \ar[l] } \\

$A_3$ & $A'_3$
\end{tabular}
\end{center}

By Definitions \ref{def11} and \ref{def12} the exchange graphs associated the quivers $\Gamma_r$ and $\Gamma'_r$ are isomorphic. Therefore, mutations can be made for $\Gamma_r$ and then in the seed of the last mutation we can replace $x_{k_0}$ by $m= a_0x_0+a_1x_1 + \cdots + a_{r-1}x_{r-1}$.

Note that if $k_i = k_{i+1}$ for any $i \in \lbrace 1,\ldots,t-1 \rbrace$, $\mu_{k_{i+1}}(\mu_{k_i}) = \mu^2_{k_i}$ and by Lemma \ref{lema1} $\mu^2_{k_i}(\tilde{x},\tilde{B})=(\tilde{x},\tilde{B})$. That is, the seed not modify; so to avoid this,  $k_1,\ldots, k_t$   are chosen  in such a way that each pair of integers $k_i,k_{i+1}$ with $i \in \lbrace 1,\ldots,t-1 \rbrace$ are different.

In the sequence $\{k_1,\ldots, k_t\}$ must appearing $k_0$ at least once, because we want that the cluster variable $x_{k_0}$ to be  modified after doing the mutation sequence. So in the place $k_0$ of the cluster will appear a new cluster variable. 

From (\ref{eq1}) in the Definition \ref{def6} and the result of the mutation of the quiver $ \Gamma_r $ at vertex $ k_0 $, we note that in the new vertex $ x'_{k_0} $ appears the cluster variables that are adjacent to $ x_{k_0} $ in the quiver $ \Gamma_r $. Thus, before $ k_0 $ in the sequence $\{ k_1, \ldots, k_t \}$ must appear at least one $ k_i $, with $ i \in \lbrace 1, \ldots, t-1 \rbrace $, such that $ x_{k_i} $ is adjacent to $ x_{k_0} $.

The encrypt message $\tilde{c}$, is obtained as part of seed:
$$(\tilde{c},\tilde{B'})=\mu_{k_{t}}(\mu_{k_{t-1}}(\cdots \mu_{k_2}(\mu_{k_1}(\tilde{x},\tilde{B})))),$$ and subsequently replacing each cluster variable $x_{k_0}$ by $m= a_0x_0+a_1x_1 + \cdots + a_{r-1}x_{r-1}$. Thus, the seed $(\tilde{c},\tilde{B'})$ is send to $B$.

\subsection{Decryption of the message}

As the $ k $ key is private, $ B $ can be manipulated to reverse the order and get
\begin{equation}
l=\lbrace k_{t}, k_{t-1}, \ldots, k_2, k_1, k_0 \rbrace,
\label{eq7}
\end{equation}
where the last number indicate the place where is saved the encrypted message.

When $B$ receives the seed $(\tilde{c},\tilde{B'})$, $B$ does the mutations sequence $k_{t}, k_{t-1}, \ldots, k_2, k_1$ and obtains the seed which contains the plaintext $m$ as stablished in the following theorem. 

\begin{Theorem}
Let $\mathcal{A}_r$ be the cluster algebra associated to the quiver $\Gamma_r$, $k={k_0,k_1,\ldots,k_t}$ as in (\ref{eq6}), $l=\lbrace k_{t}, k_{t-1}, \ldots, k_2, k_1, k_0 \rbrace$ as in (\ref{eq7}) and let $(\tilde{x},\tilde{B})$ be  the initial seed of $\mathcal{A}_r$. If $(\tilde{c},\tilde{B'})=\mu_{k_{t}}(\mu_{k_{t-1}}(\cdots \mu_{k_2}(\mu_{k_1}(\tilde{x},\tilde{B}))))$ then $(\tilde{x},\tilde{B})= \mu_{k_{1}}(\mu_{k_{2}}(\cdots \mu_{k_{t-1}}(\mu_{k_t}(\tilde{c},\tilde{B'}))))$. 
\label{teo10}
\end{Theorem}

\begin{proof}
By Lemma \ref{lema1}

\begin{equation*}
\begin{array}{rl}
\mu_{k_{1}}(\mu_{k_{2}}(\cdots \mu_{k_{t-1}}(\mu_{k_t}(\tilde{c},\tilde{B'})))) & = \mu_{k_{1}}(\mu_{k_{2}}(\cdots \mu_{k_{t-1}}(\mu^2_{k_t}(\mu_{k_{t-1}}(\cdots \mu_{k_2}(\mu_{k_1}(\tilde{x},\tilde{B}))))))) \\
& = \mu_{k_{1}}(\mu_{k_{2}}(\cdots \mu^2_{k_{t-1}}(\cdots \mu_{k_2}(\mu_{k_1}(\tilde{x},\tilde{B}))))) \\
& = \vdots \\
& = \mu_{k_{1}}(\mu^2_{k_{2}}(\mu_{k_1}(\tilde{x},\tilde{B}))) \\
& = \mu^2_{k_{1}}(\tilde{x},\tilde{B}) \\
& = (\tilde{x},\tilde{B}).
\end{array}
\end{equation*}

\end{proof}

By last, the cluster variable in the place $k_0$ is recuperated the cluster variable in the place $k_0$, and after replacing each $x_i$ by $\alpha^i,$ the message $m$ is obtained  as an element in a finite field $\mathbb{F}_q$.

\section{Examples}

In this section we present two examples to encrypt and decrypt a message using cluster algebras of finite type. First we encrypt a letter and in the second example we encrypt a message seen as a number.



\subsection*{Example 1}

For this example enumerate the letters of the alphabet with numbers, obtaining the following table.

\begin{center}
\begin{tabular}{||c|c|c|c|c|c|c|c|c|c|c|c|c||}
\hline \hline
A & B & C & D & E & F & G & H & I & J & K & L & M \\ 
\hline
1 & 2 & 3 & 4 & 5 & 6 & 7 & 8 & 9 & 10 & 11 & 12 & 13 \\ 
\hline \hline
 N & O & P & Q & R & S & T & U & V & W & X & Y & Z \\
\hline
14 & 15 & 16 & 17 & 18 & 19 & 20 & 21 & 22 & 23 & 24 & 25 & 26\\
\hline \hline 
\end{tabular}

\vspace{0.3cm}

Table 1

\end{center}

To choose $p$ and $r$, we consider a number prime $p$ and $r \in \mathbb{N}$ such that $p^r \geq 26$ which is the number of letters in the alphabet that we are going to use. In this case our choose $p=2$ and $r=5$, that is, we use the finite field $\mathbb{F}_{2^5}$. The irreducible polynomial of degree $5$ on $\mathbb{Z}_2$ that  we consider is $f(x)=x^5+x^2+1$. Since $32$ exceeds the number of letters in the alphabet, we complete the Table 1 with the following numbers ans letters (the letters can be  choosen randomly of the alphabet). 

\begin{center}
\begin{tabular}{||c|c|c|c|c||}
\hline \hline
X & Y & Z & X & Y   \\ 
\hline
27 & 28 & 29 & 30 & 31  \\ 
\hline \hline 
\end{tabular}

\vspace{0.3cm}

Tabla 2

\end{center}

The quiver we consider for the cluster algebra $\mathcal{A}_5$ with the Dynkin diagram $A_5$ is

\begin{center}
\hspace{0.5cm}
\xymatrix{
x_0 \ar[r] & x_1  & x_2 \ar[l]  \ar[r] & x_3 & x_4 \ar[l].}
\end{center}

So the exchange matrix associated is 

\begin{equation*}
\tilde{B}=\left( 
\begin{array}{ccccc}
0 & 1 & 0 & 0 & 0 \\
-1 & 0 & -1 & 0 & 0 \\
0 & 1 & 0 & 1 & 0 \\
0 & 0 & -1 & 0 & -1 \\
0 & 0 & 0 & 1 & 0
\end{array}
\right) 
\end{equation*}

\vspace{1cm}

Suppose $A$  sends the message $m$ (seen as the letter $F$) to  $B$, using the key $k=\{ 0,1,4,0,3,1 \}$.

\vspace{0.5cm}

\textbf{Encryption}

\vspace{0.5cm}

According to Table 1 the letter $F$ is assigned the number $6$. The binary representation of $F$ is $$0\cdot 2^0 + 1 \cdot 2^1 + 1\cdot 2^2,$$ so that $m=01100 \in \mathbb{F}_{2^5} = \mathbb{Z}_2[x]/\langle x^5 + x^2 + 1 \rangle$ seen as $\mathbb{Z}_2$ vector space of dimension 5.

So $m$ can represented as $x_1 + x_2$, where $x_1$ and $x_2$ are cluster variables of the initial seed $\tilde{x}=\lbrace x_0,x_1,x_2,x_3,x_4 \rbrace$. Since $k_0=0$ in the key $k$, then the label for the vertices in the quiver $A'_5$ is $\lbrace x_1 + x_2,x_1,x_2,x_3,x_4 \rbrace$.



To encrypt the message, we do the sequence of mutations $\mu_1$, $\mu_4$ , $\mu_0$, $\mu_3$, $\mu_1$. The results of the sequence mutations in $A_5$ have been obtained with Sage 7.6 (cf. \cite{Sage}) and show the following.

\begin{verbatim}
In[1]: S1=ClusterSeed(['A',5])
Q1=S1.quiver()
In[2]: ms=S1.mutation_sequence([1,4,0,3,1],return_output='var')
print ms
Out[2]:
[(x0*x2 + 1)/x1, (x3 + 1)/x4, (x0*x2 + x1 + 1)/(x0*x1), 
(x2*x4 + x3 + 1)/(x3*x4), (x1 + 1)/x0]
In[3]: print S1.mutation_sequence([1,4,0,3,1],return_output='matrix')
Out[3]:
[[ 0 -1  0  0  0]
[ 1  0  1  0  0]
[ 0 -1  0  1  0]
[ 0  0 -1  0 -1]
[ 0  0  0  1  0], [ 0 -1  0  0  0]
[ 1  0  1  0  0]
[ 0 -1  0  1  0]
[ 0  0 -1  0  1]
[ 0  0  0 -1  0], [ 0  1  0  0  0]
[-1  0  1  0  0]
[ 0 -1  0  1  0]
[ 0  0 -1  0  1]
[ 0  0  0 -1  0], [ 0  1  0  0  0]
[-1  0  1  0  0]
[ 0 -1  0 -1  1]
[ 0  0  1  0 -1]
[ 0  0 -1  1  0], [ 0 -1  1  0  0]
[ 1  0 -1  0  0]
[-1  1  0 -1  1]
[ 0  0  1  0 -1]
[ 0  0 -1  1  0]] 
In[4]: Q1.mutation_sequence([1,4,0,3,1],show_sequence=True)
Out[4]:
\end{verbatim}

\includegraphics[scale=0.5]{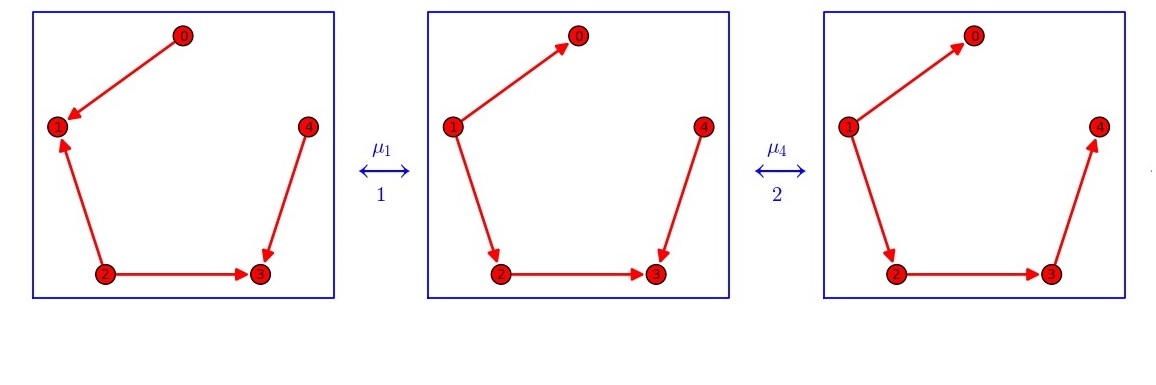} 

\includegraphics[scale=0.5]{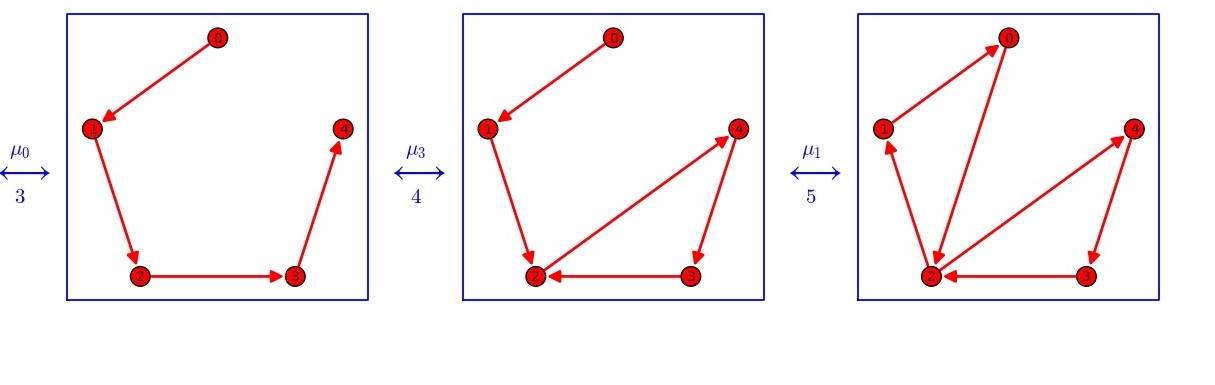} 

In Out[2] appeasr the cluster variables which change in each mutation. So the seed at end of the sequence of mutations is $(\tilde{y},\tilde{B'})$, where
$$\tilde{y}=\lbrace \dfrac{x_0 x_2 + x_1 + 1}{x_0 x_1}, \dfrac{x_1 + 1 }{x_0}, x_2, \dfrac{x_2 x_4 + x_3 + 1}{x_3 x_4}, \dfrac{x_3 + 1}{x_4} \rbrace,$$ and
\begin{equation}
\tilde{B'}=\left( 
\begin{array}{ccccc}
0 & -1 & 1 & 0 & 0 \\
1 & 0 & -1 & 0 & 0 \\
-1 & 1 & 0 & -1 & 1 \\
0 & 0 & 1 & 0 & -1 \\
0 & 0 & -1 & 1 & 0
\end{array}
\right). 
\label{eq8}
\end{equation}

Replacing $m=x_1 + x_2$ in $x_0$, we obtain

\begin{verbatim}
In [5]: [cv.subs(x0=S1.x(1)+S1.x(2),x1=S1.x(1),x2=S1.x(2),x3=S1.x(3),
x4=S1.x(4)) for cv in ms]
Out[5]:
[(x1*x2 + x2^2 + 1)/x1, (x3 + 1)/x4,
(x1*x2 + x2^2 + x1 + 1)/(x1^2 + x1*x2),(x2*x4 + x3 + 1)/(x3*x4),
(x1 + 1)/(x1 + x2)]
\end{verbatim}

We have the seed $(\tilde{c},\tilde{B'})$ in the cluster algebra $\mathcal{A}_5$ associated with the Dynkin diagram $A'_5$, where $\tilde{B'}$ is as in (\ref{eq8}) and the cluster $\tilde{c}$ is
$$\tilde{c}=\lbrace \dfrac{x_1 x_2 + x_2^2 + x_1 + 1}{x_1^2 + x_1x_2}, \dfrac{x_1 + 1 }{x_1 + x_2}, x_2, \dfrac{x_2 x_4 + x_3 + 1}{x_3 x_4}, \dfrac{x_3 + 1}{x_4} \rbrace.$$

Since $x_i$ can be seen as $\alpha^i$ by the Proposition \ref{prop1}, then

$$ \dfrac{x_1 x_2 + x_2^2 + x_1 + 1}{x_1^2 + x_1x_2} = \frac{\alpha^3 + \alpha^4 + \alpha + 1}{\alpha^2 + \alpha^3}= 1 + \alpha + \alpha^3 = 11010$$

$$\dfrac{x_1 + 1 }{x_1 + x_2} = \frac{\alpha+ 1}{\alpha + \alpha^2}= \alpha + \alpha^4 = 01001$$

$$x_2 = \alpha^2 = 00100 $$

$$\dfrac{x_2 x_4 + x_3 + 1}{x_3 x_4}= \frac{\alpha^6 + \alpha^3 + 1}{\alpha^7}=1 + \alpha + \alpha^2= 11100$$

$$\dfrac{x_3 + 1}{x_4}= \frac{\alpha^3 + 1}{\alpha^4}= 1 + \alpha^3 + \alpha^4 = 10011$$

So $$\tilde{c}= \{ 11010,01001,00100,11100,10011 \} = \{ 11,18,4,7,25\} = \{ L,S,E,H,Z \}).$$ In this way $B$ receives $(\tilde{c},\tilde{B'})= (\{ L,S,E,H,Z \}, \tilde{B'})$ and the message $m=F$ is encrypted as the letter $L$.

\vspace{0.47cm}

\textbf{Decrypt}

\vspace{0.47cm}

The decrypt key of $B$ is $l=\{ 1,3,0,4,1,0 \}$. The last number, 0, tells us the place where the encrypted message is saved. As $B$ receives 
$$(\tilde{c},\tilde{B'})= (\{ L,S,E,H,Z \}, \tilde{B'}),$$ where we see each letter of $\tilde{c}$  as an element of the finite field $ \mathbb{F}_{2^5}$, leaving 
$$\tilde{c}=\lbrace 11010,01001,00100,11100,10011 \rbrace,$$ and seen as a linear combination of elements of the cluster $\tilde{x}  = \{ x_0,x_1,x_2,x_3,x_4 \}$, we have
$$\tilde{c}=\lbrace x_0 + x_1 + x_3,x_1 + x_4,x_2,x_0 + x_1 + x_2,x_0 + x_3 + x_4 \rbrace.$$

Doing the sequence of mutations $\mu_1,\mu_3,\mu_0,\mu_4,\mu_1$ at the seed $(\tilde{x},\tilde{B'})$ in the cluster algebra associated with the matrix $\tilde{B'}$. We obtain

\begin{verbatim}
In[6]: G=matrix([[0,-1,1,0,0],[1,0,-1,0,0],[-1,1,0,-1,1],
[0,0,1,0,-1],[0,0,-1,1,0]])
S1=ClusterSeed(G)
Q1=S1.quiver()
In[7]: ms=S1.mutation_sequence([1,3,0,4,1],return_output='var')
print ms
print S1.mutation_sequence([1,3,0,4,1],return_output='matrix')
Q1.mutation_sequence([1,3,0,4,1],show_sequence=True)
Out[7]:
[(x0 + x2)/x1, (x2 + x4)/x3, (x0 + x1 + x2)/(x0*x1), 
(x2 + x3 + x4)/(x3*x4), (x1 + x2)/x0]
[[ 0  1  0  0  0]
[-1  0  1  0  0]
[ 0 -1  0 -1  1]
[ 0  0  1  0 -1]
[ 0  0 -1  1  0], [ 0  1  0  0  0]
[-1  0  1  0  0]
[ 0 -1  0  1  0]
[ 0  0 -1  0  1]
[ 0  0  0 -1  0], [ 0 -1  0  0  0]
[ 1  0  1  0  0]
[ 0 -1  0  1  0]
[ 0  0 -1  0  1]
[ 0  0  0 -1  0], [ 0 -1  0  0  0]
[ 1  0  1  0  0]
[ 0 -1  0  1  0]
[ 0  0 -1  0 -1]
[ 0  0  0  1  0], [ 0  1  0  0  0]
[-1  0 -1  0  0]
[ 0  1  0  1  0]
[ 0  0 -1  0 -1]
[ 0  0  0  1  0]]
\end{verbatim}

\includegraphics[scale=0.47]{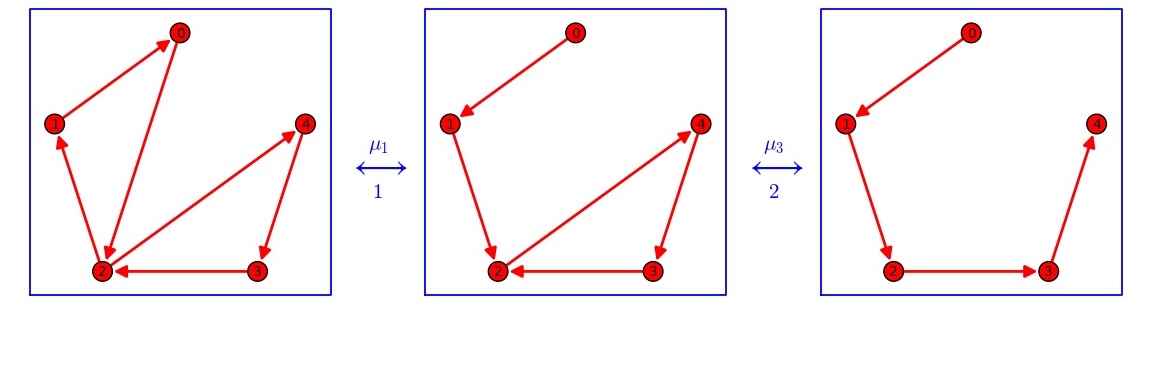}

\includegraphics[scale=0.47]{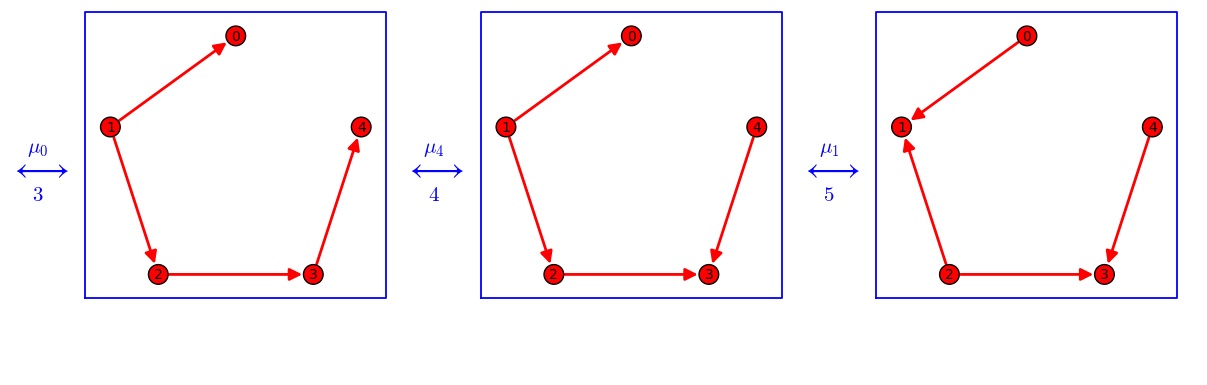}  

Replacing $x_0 = x_0 + x_1$, $x_1 = x_1 + x_4$, $x_2 = x_2$, $x_3 = x_0 + x_1 + x_2$, $x_4= x_0 + x_3 + x_4$ is obtain

\begin{verbatim}
In[8]: [cv.subs(x0=S1.x(0)+S1.x(1)+S1.x(3),x1=S1.x(1)+S1.x(4),x2=S1.x(2),
x3=S1.x(0)+S1.x(1)+S1.x(2),x4=S1.x(0)+S1.x(3)+S1.x(4)) for cv in ms]
Out[8]:
[(x0 + x1 + x2 + x3)/(x1 + x4), (x0 + x2 + x3 + x4)/(x0 + x1 + x2),
(x0 + 2*x1 + x2 + x3 + x4)/(x0*x1 + x1^2 + x1*x3 + x0*x4 + x1*x4 + x3*x4),
(2*x0 + x1 + 2*x2 + x3 + x4)/(x0^2 + x0*x1 + x0*x2 + x0*x3 + x1*x3 + 
x2*x3 + x0*x4 + x1*x4 + x2*x4), (x1 + x2 + x4)/(x0 + x1 + x3)]

\end{verbatim}

Recall that the clear text is in place 0, so we are only interested in the cluster variable in that place; that is $\dfrac{x_0 + 2x_1 + x_2 + x_3 + x_4}{x_0x_1 + x_1^2 + x_1x_3 + x_0x_4 + x_1x_4+x_3x_4}$, which are seen as elements of  $\mathbb{F}_{2^5}$ and by Proposition \ref{prop1} we have
$$\dfrac{1+2\alpha + \alpha^2 + \alpha^3 + \alpha^4}{\alpha + \alpha^2 + \alpha^4 + \alpha^4 + \alpha^5 + \alpha^7}= \alpha + \alpha^2 = 01100,$$ which corresponds to the number 6, and according to Table 1 corresponds to the letter $F$; that is, the original message that $ A $ sent to $ B $ was retrieved.

\subsection*{Example 2}

For this example we consider the message $m=38927$, $p=101$, $r=7$ and the irreducible polynomial of degree 7 on $\mathbb{Z}_{101}$ is $f(x)=46+x^2+x^3+74x^5+x^7$. So we use the finite field $\mathbb{F}_{101^7}$. Suppose $A$ is the entity who want send the message $m=38927$ at entity $B$ using the key $k=\{ 3,2,3,4,3 \}$. 

We consider the Dynkin diagramm $D_7$ and the quiver we consider for the cluster algebra $\mathcal{A}_7$ with the Dynkin diagram $D_7$ is

\begin{center}
\hspace{0.5cm}
\xymatrix{ &      &                    &     &  x_6  &    \\        
x_0 \ar[r] & x_1  & x_2 \ar[l]  \ar[r] & x_3 & x_4 \ar[l] \ar@{->}[u] \ar[r]& x_5,}
\end{center} and exchange matrix associated is

\begin{equation*}
\tilde{B}=\left( 
\begin{array}{ccccccc}
0 & 1 & 0 & 0 & 0 & 0 & 0 \\
-1 & 0 & -1 & 0 & 0 & 0 & 0\\
0 & 1 & 0 & 1 & 0 & 0 & 0\\
0 & 0 & -1 & 0 & -1 & 0 & 0\\
0 & 0 & 0 & 1 & 0 & 1 & 1 \\
0 & 0 & 0 & 0 & -1 & 0 & 0 \\
0 & 0 & 0 & 0 & -1 & 0 & 0 
\end{array}
\right) 
\end{equation*}

The representation in base 101 of $m$ is $$m=42\cdot101^0+82\cdot101^1+3\cdot101^2,$$ so that $m=( 42,82,3,0,0,0,0 ) \in \mathbb{F}_{101^7}$ seen as $\mathbb{Z}_{101}$ vector space of dimension 7. So $m = 42x_0 + 82x_1+3x_2$ is a linear combination of cluster variables in the initial seed with initial cluster $\tilde{x}=\{ x_0,x_1,x_2,x_3,x_4,x_5,x_6 \}$. Since $k_0=3$ in the key $k$, then the label for the vertices in the quiver $D'_7$ is $\{ x_0,x_1,x_2,42x_0 + 82x_1+3x_2,x_4,x_5,x_6 \}$.

To encrypt the message, do the sequence of mutations $\mu_2$, $\mu_3$ , $\mu_4$, $\mu_3$. The results of the sequence mutations in $\mathcal{A}_7$ show following.

\begin{verbatim}
In [1]:
S2=ClusterSeed(['D',7])
Q2=S2.quiver()
In [2]:
ms=S2.mutation_sequence([2,3,4,3],return_output='var')
print ms
Out[2]:
[(x1*x3 + 1)/x2, (x1*x3 + x2*x4 + 1)/(x2*x3), (x1*x3^2*x5*x6 + x3*x5*x6 
+ x1*x3 + x2*x4 + 1)/(x2*x3*x4), (x3*x5*x6 + 1)/x4]
In [3]:
print S2.mutation_sequence([2,3,4,3],return_output='matrix')
Out[3]:
[[ 0  1  0  0  0  0  0]
[-1  0  1  0  0  0  0]
[ 0 -1  0 -1  0  0  0]
[ 0  0  1  0 -1  0  0]
[ 0  0  0  1  0  1  1]
[ 0  0  0  0 -1  0  0]
[ 0  0  0  0 -1  0  0], [ 0  1  0  0  0  0  0]
[-1  0  1  0  0  0  0]
[ 0 -1  0  1 -1  0  0]
[ 0  0 -1  0  1  0  0]
[ 0  0  1 -1  0  1  1]
[ 0  0  0  0 -1  0  0]
[ 0  0  0  0 -1  0  0], [ 0  1  0  0  0  0  0]
[-1  0  1  0  0  0  0]
[ 0 -1  0  0  1  0  0]
[ 0  0  0  0 -1  1  1]
[ 0  0 -1  1  0 -1 -1]
[ 0  0  0 -1  1  0  0]
[ 0  0  0 -1  1  0  0], [ 0  1  0  0  0  0  0]
[-1  0  1  0  0  0  0]
[ 0 -1  0  0  1  0  0]
[ 0  0  0  0  1 -1 -1]
[ 0  0 -1 -1  0  0  0]
[ 0  0  0  1  0  0  0]
[ 0  0  0  1  0  0  0]]
In [4]:
Q2.mutation_sequence([2,3,4,3],show_sequence=True)
\end{verbatim}

\begin{center}
\includegraphics[scale=0.22]{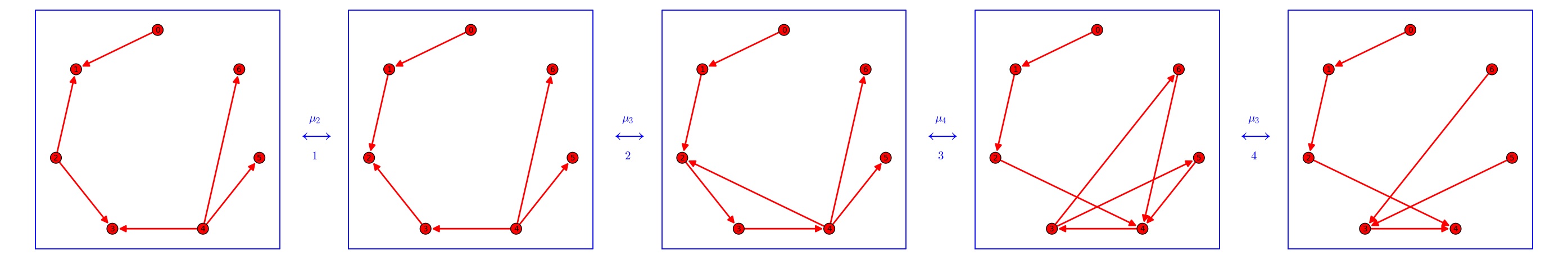} 
\end{center}

The seed at end of the sequence mutations is $(\tilde{y},\tilde{B'})$ where

$$\tilde{y}=\lbrace x_0,x_1,\frac{x_1x_3+1}{x_2},\frac{x_3x_5x_6+1}{x_4},\frac{x_1x_3^2x_5x_6+x_3x_5x_6+x_1x_3+x_2x_4+1}{x_2x_3x_4},x_5,x_6 \rbrace,$$ and

\begin{equation}
\tilde{B'}=\left( 
\begin{array}{ccccccc}
0 & 1 & 0 & 0 & 0 & 0 & 0 \\
-1 & 0 & 1 & 0 & 0 & 0 & 0\\
0 & -1 & 0 & 0 & 1 & 0 & 0\\
0 & 0 & 0 & 0 & 1 & -1 & -1\\
0 & 0 & -1 & -1 & 0 & 0 & 0 \\
0 & 0 & 0 & 1 & 0 & 0 & 0 \\
0 & 0 & 0 & 1 & 0 & 0 & 0 
\end{array}
\right).
\label{eq9}
\end{equation}

Replacing $m=42x_0 + 82x_1+3x_2$ in $x_3$, we obtain

\begin{equation*}
\begin{array}{l}
\tilde{c}=\lbrace x_0, x_1, (42x_0x_1+82x_1^2+3x_1x_2+1)/x_2, (42x_0x_5x_6+82x_1x_5x_6 + 3x_2x_5x_6+1)/x_4, \\

(1764x_0^2x_1x_5x_6+6888x_0x_1^2x_5x_6+6724x_1^3x_5x_6 +252x_0x_1x_2x_5x_6 +492x_1^2x_2x_5x_6 \\
+9x_1x_2^2x_5x_6 +42x_0x_5x_6+82x_1x_5x_6 +3x_2x_5x_6 +42x_0x_1+82x_1^2+3x_1x_2+x_2x_4+1)/ \\
(42x_0x_2x_4+82x_1x_2x_4+3x_2^2x_4), x_5, x_6 \rbrace.
\end{array} 
\end{equation*}

Since $x_i$ can be  seen as $\alpha^i$ by Proposition \ref{prop1}, then
\begin{equation*}
\begin{array}{l}
\tilde{c} =\lbrace (1,0,0,0,0,0,0), (0,1,0,0,0,0,0), (71,35,43,95,51,90,43), (18,100,62,52,82,5,12), \\(11,41,9,94,83,52,55),(0,0,0,0,0,1,0), (0,0,0,0,0,0,1)  \rbrace \\
 = \lbrace 1,101,46 596 680 922 228 , 12 799 379 480 831,58 938 867 466 645, 10 510 100 501, 1 061 520 150 601 \rbrace.
\end{array} 
\end{equation*}

So the message $m=38927$ is encrypted as the number $12 799 379 480 831$. 

When $B$ receives $(\tilde{c},\tilde{B'})$ where $\tilde{B'}$ as in \ref{eq9}, $B$ will retrieve the message using the key $l= \{ 3,4,3,2,3 \}$ and the last number 3 in $l$ indicate the place where the encrypted message is saved. Each number of $\tilde{c}$ can be seen as an element of the finite field $\mathbb{F}_{101^7}$ and later seen as linear combination of elements of the cluster $\tilde{x}= \{ x_0,x_1,x_2,x_3,x_4,x_5,x_6 \}$ we have 
\begin{equation*}
\begin{array}{l}
\tilde{c} =  \{ x_0,x_1,71x_0+35x_1+43x_2+95x_3+51x_4+90x_5+43x^6,18x_0+100x_1+62x_2+52x_3 \\
+82x_4+5x_5+12x_6,11x_0+41x_1+9x_2+94x_3+83x_4+52x_5+55x_6,x_5,x_6  \}.
\end{array} 
\end{equation*}

Doing the sequence of mutations $\mu_3,\mu_4,\mu_3,\mu_2$ at the seed $(\tilde{x},\tilde{B'})$ in the cluster algebra associated with the matrix $\tilde{B'}$. We obtain

\begin{verbatim}
In[5]:
G=matrix([[0,1,0,0,0,0,0],[-1,0,1,0,0,0,0],[0,-1,0,0,1,0,0],
[0,0,0,0,1,-1,-1],[0,0,-1,-1,0,0,0],[0,0,0,1,0,0,0],[0,0,0,1,0,0,0]])
S3=ClusterSeed(G)
Q3=S3.quiver()
In[6]:
ms=S3.mutation_sequence([3,4,3,2],return_output='var')
print ms
print S3.mutation_sequence([3,4,3,2],return_output='matrix')
Q3.mutation_sequence([3,4,3,2],show_sequence=True)
Out[6]:
[(x5*x6 + x4)/x3, (x2*x3*x5*x6 + x5*x6 + x4)/(x3*x4), (x2*x3 + 1)/x4,
 (x1*x2*x3 + x1 + x4)/(x2*x4)]
[[ 0  1  0  0  0  0  0]
[-1  0  1  0  0  0  0]
[ 0 -1  0  0  1  0  0]
[ 0  0  0  0 -1  1  1]
[ 0  0 -1  1  0 -1 -1]
[ 0  0  0 -1  1  0  0]
[ 0  0  0 -1  1  0  0], [ 0  1  0  0  0  0  0]
[-1  0  1  0  0  0  0]
[ 0 -1  0  1 -1  0  0]
[ 0  0 -1  0  1  0  0]
[ 0  0  1 -1  0  1  1]
[ 0  0  0  0 -1  0  0]
[ 0  0  0  0 -1  0  0], [ 0  1  0  0  0  0  0]
[-1  0  1  0  0  0  0]
[ 0 -1  0 -1  0  0  0]
[ 0  0  1  0 -1  0  0]
[ 0  0  0  1  0  1  1]
[ 0  0  0  0 -1  0  0]
[ 0  0  0  0 -1  0  0], [ 0  1  0  0  0  0  0]
[-1  0 -1  0  0  0  0]
[ 0  1  0  1  0  0  0]
[ 0  0 -1  0 -1  0  0]
[ 0  0  0  1  0  1  1]
[ 0  0  0  0 -1  0  0]
[ 0  0  0  0 -1  0  0]]
\end{verbatim}

\begin{center}
\includegraphics[scale=0.22]{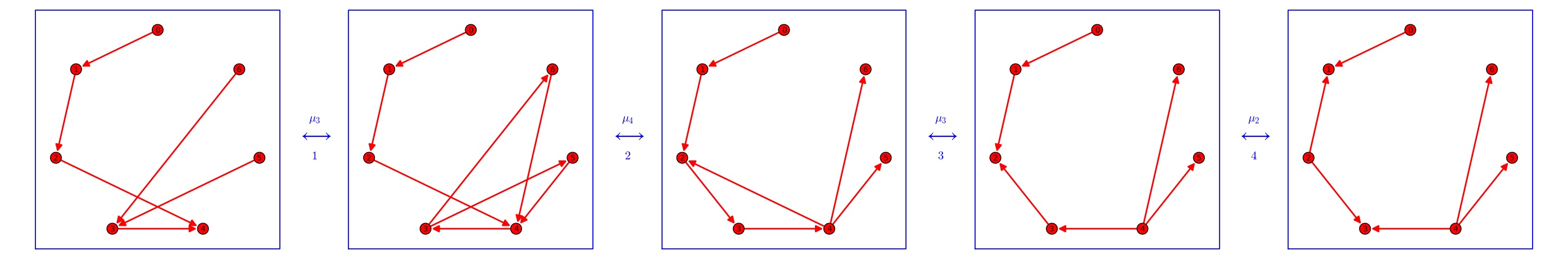} 
\end{center}

Replacing $x_0=x_0$, $x_1=x_1$, $x_2= 71x_0+35x_1+43x_2+95x_3+51x_4+90x_5 + 43x_6$, $x_3=18x_0+100x_1+62x_2+52x_3+82x_4+5x_5+12x_6$, $x_4=11x_0 + 41x_1+9x_2+94x_3+83x_4+52x_5+55x_6$, $x_5 = x_5$ and $x_6=x_6$ is obtain

\begin{verbatim}
In[7]:
[cv.subs(x0=S3.x(0),x1=S3.x(1),x2=71*S3.x(0)+35*S3.x(1)+43*S3.x(2)
+95*S3.x(3)+51*S3.x(4)+90*S3.x(5)+43*S3.x(6),x3=18*S3.x(0)+100*S3.x(1)
+62*S3.x(2)+52*S3.x(3)+82*S3.x(4)+5*S3.x(5)+12*S3.x(6),x4=11*S3.x(0)
+41*S3.x(1)+9*S3.x(2)+94*S3.x(3)+83*S3.x(4)+52*S3.x(5)+55*S3.x(6),
x5=S3.x(5),x6=S3.x(6)) for cv in ms]
Out[7]:
[(x5*x6 + 11*x0 + 41*x1 + 9*x2 + 94*x3 + 83*x4 + 52*x5 + 55*x6)/(18*x0 
+ 100*x1 + 62*x2 + 52*x3 + 82*x4 + 5*x5 + 12*x6),
(1278*x0^2*x5*x6 + 7730*x0*x1*x5*x6 + 3500*x1^2*x5*x6 + 5176*x0*x2*x5*x6 
+ 6470*x1*x2*x5*x6 + 2666*x2^2*x5*x6 + 5402*x0*x3*x5*x6 +
11320*x1*x3*x5*x6+ 8126*x2*x3*x5*x6 + 4940*x3^2*x5*x6 + 6740*x0*x4*x5*x6
+ 7970*x1*x4*x5*x6 + 6688*x2*x4*x5*x6 + 10442*x3*x4*x5*x6 + 
4182*x4^2*x5*x6 + 1975*x0*x5^2*x6 + 9175*x1*x5^2*x6 + 5795*x2*x5^2*x6 + 
5155*x3*x5^2*x6 + 7635*x4*x5^2*x6 + 450*x5^3*x6 + 1626*x0*x5*x6^2 + 
4720*x1*x5*x6^2 + 3182*x2*x5*x6^2 + 3376*x3*x5*x6^2 + 4138*x4*x5*x6^2 + 
1295*x5^2*x6^2 + 516*x5*x6^3 + x5*x6 + 11*x0 + 41*x1 + 9*x2 + 94*x3 + 
83*x4 + 52*x5 + 55*x6)/(198*x0^2 + 1838*x0*x1+ 4100*x1^2 + 844*x0*x2 + 
3442*x1*x2 + 558*x2^2 + 2264*x0*x3 + 11532*x1*x3 + 6296*x2*x3 + 4888*x3^2 
+ 2396*x0*x4 + 11662*x1*x4 + 5884*x2*x4 + 12024*x3*x4 + 6806*x4^2 + 
991*x0*x5 + 5405*x1*x5 + 3269*x2*x5 + 3174*x3*x5 + 4679*x4*x5 + 260*x5^2 
+ 1122*x0*x6 + 5992*x1*x6 + 3518*x2*x6 + 3988*x3*x6 + 5506*x4*x6 + 
899*x5*x6 + 660*x6^2),
(1278*x0^2 + 7730*x0*x1 + 3500*x1^2 + 5176*x0*x2 + 6470*x1*x2 + 2666*x2^2
+ 5402*x0*x3 + 11320*x1*x3 + 8126*x2*x3 + 4940*x3^2 + 6740*x0*x4 + 
7970*x1*x4 + 6688*x2*x4 + 10442*x3*x4 + 4182*x4^2 + 1975*x0*x5 + 
9175*x1*x5 + 5795*x2*x5 + 5155*x3*x5 + 7635*x4*x5 + 450*x5^2 + 1626*x0*x6 
+ 4720*x1*x6 + 3182*x2*x6 + 3376*x3*x6 + 4138*x4*x6 + 1295*x5*x6 + 
516*x6^2 + 1)/(11*x0 + 41*x1 + 9*x2 + 94*x3 + 83*x4 + 52*x5 + 55*x6),
(1278*x0^2*x1 + 7730*x0*x1^2 + 3500*x1^3 + 5176*x0*x1*x2 + 6470*x1^2*x2 + 
2666*x1*x2^2 + 5402*x0*x1*x3 + 11320*x1^2*x3 + 8126*x1*x2*x3 + 
4940*x1*x3^2 + 6740*x0*x1*x4 + 7970*x1^2*x4 + 6688*x1*x2*x4 + 
10442*x1*x3*x4 + 4182*x1*x4^2 + 1975*x0*x1*x5 + 9175*x1^2*x5 + 
5795*x1*x2*x5 + 5155*x1*x3*x5 + 7635*x1*x4*x5 + 450*x1*x5^2 + 
1626*x0*x1*x6 + 4720*x1^2*x6 + 3182*x1*x2*x6 + 3376*x1*x3*x6 + 
4138*x1*x4*x6 + 1295*x1*x5*x6 + 516*x1*x6^2 + 11*x0 + 42*x1 + 9*x2 + 
94*x3 + 83*x4+ 52*x5 + 55*x6)/(781*x0^2 + 3296*x0*x1 + 1435*x1^2 + 
1112*x0*x2 + 2078*x1*x2 + 387*x2^2 + 7719*x0*x3 + 7185*x1*x3 + 4897*x2*x3 
+ 8930*x3^2 + 6454*x0*x4 + 4996*x1*x4+ 4028*x2*x4 + 12679*x3*x4 + 
4233*x4^2 + 4682*x0*x5 + 5510*x1*x5 + 3046*x2*x5 +13400*x3*x5 + 
10122*x4*x5 + 4680*x5^2 + 4378*x0*x6 + 3688*x1*x6 + 2752*x2*x6 + 
9267*x3*x6 + 6374*x4*x6 + 7186*x5*x6 + 2365*x6^2)]
\end{verbatim}

Remember that the clear text is in place 3, so we are only interested in the cluster variable in this place, that is
\begin{verbatim}
(1278*x0^2 + 7730*x0*x1 + 3500*x1^2 + 5176*x0*x2 + 6470*x1*x2 + 2666*x2^2
+ 5402*x0*x3 + 11320*x1*x3 + 8126*x2*x3 + 4940*x3^2 + 6740*x0*x4 + 
7970*x1*x4 + 6688*x2*x4 + 10442*x3*x4 + 4182*x4^2 + 1975*x0*x5 + 
9175*x1*x5 + 5795*x2*x5 + 5155*x3*x5 + 7635*x4*x5 + 450*x5^2 + 1626*x0*x6 
+ 4720*x1*x6 + 3182*x2*x6 + 3376*x3*x6 + 4138*x4*x6 + 1295*x5*x6 + 
516*x6^2 + 1)/(11*x0 + 41*x1 + 9*x2 + 94*x3 + 83*x4 + 52*x5 + 55*x6)
\end{verbatim}

This cluster variable seen as element of $\mathbb{F}_{101^7}$ and by Proposition \ref{prop1} we have

\begin{equation*}
\begin{array}{rl}
\frac{47 + 84 \alpha + 16 \alpha^2 + 79 \alpha^3 + 47 \alpha^4 + 15 \alpha^5 + 67 \alpha^6}{11 + 41\alpha + 9 \alpha^2 + 94 \alpha^3 + 83 \alpha^4 + 52 \alpha^5 + 55 \alpha^6} & = 42+82\alpha+3\alpha^2 \\
                                                   & = (42,82,3,0,0,0,0) \\
                                                   & = 38927 \\
\end{array} 
\end{equation*}

Thus $B$ retrieve the original message $m$.
 
\section{Cryptosystem Security}

Recall that the encryption key is formed by a sequence of non-negative integers 
$$\{k_0,\ldots,k_t\}$$ where $k_0,\ldots,k_t$ is the sequence of mutations applied to the initial seed
$$\tilde{y}=\{ x_0,\ldots,x_{k_0-1},x_k=m,x_{k+1},\ldots,x_{r-1}\}.$$  When sending $(\tilde{c},\tilde{B}')$  to $B$, we know that this information can be made public, so it might be possible to know what type of Dinkin diagram $\Gamma_r$ associated with cluster algebra $A_r$ is being used. What is unknown is the place $k_0$ where the message is hidden along with the order of the mutations, which would be necessary to know if we want to recover the initial seed in order to recover the original message.

Since the exchange graph of a cluster algebra, according to Definition \ref{exchgraph} $A_r$, is a $r$-regular graph, in which the vertices are the seeds of the set $S$, so that any two of its vertices (seeds) are joined by an edge whenever there is a mutation from one seed to another.

Note that we use finite cluster algebras in this work, so  the seed set $S$ is finite \cite[Theorem 5.6.1]{Marsh}, and in this case, the complete list of seeds, which are equivalence classes of the mutation classes, could be known. However, knowing the path (sequence of mutations) to get from one seed to another depends on choosing the most suitable path from the total number of paths in the graph.

Recall that the exchange graph of a cluster algebra is $r$-regular, it can be represented by its adjacency matrix  $M$ of size $N_C$. To find the number of paths from one vertex to another, it can be done by calculating $M^t$ if the length $t$ of the path is known, or a DFS (depth-first search) algorithm can be used, which gives us all possible routes and enumerates them.

It is  worth mentioning that calculating $M^t$ has a complexity of $O(N_C^3\cdot \log t), O(t\cdot N_C^3), O(N_C^{2.807})$, or $O(N_C^w\cdot \log t)$, where $2\le w<3$, depending on the algorithm used. If a DFS algorithm is known, the complexity is $O(3\cdot N_C)$; note that this algorithm is useful when vertices are not repeated. In any of these methods used to determine the paths between vertices, it is necessary to know the length $t$ of the private key.

If the length $t$ of the private key is known, it is still not easy to determine the initial seed $\tilde{y}$, since although it would be certain that $\tilde{y}$ is in one of the mutation classes, it is not possible to know which one corresponds to the initial seed, since in a mutation class the seeds have a different order. For example, for the Dynkin diagram $A_3$ there are $N_c=14$ mutation classes and each class has $3!=6$ seeds. It can be verified that the following seeds are the mutation classes:

\begin{itemize}
    \item[] $C_1 = \left( \{x_0, x_1, x_2\}, \begin{pmatrix} 0 & 1 & 0 \\ -1 & 0 & 1 \\ 0 & -1 & 0 \end{pmatrix} \right)$
    \item[] $C_2 = \left( \{x_0, x_1, \frac{x_1+1}{x_2}\}, \begin{pmatrix} 0 & 1 & 0 \\ -1 & 0 & -1 \\ 0 & 1 & 0 \end{pmatrix} \right)$
    \item[] $C_3 = \left( \{x_0, \frac{x_0x_2+1}{x_1}, x_2\}, \begin{pmatrix} 0 & -1 & 0 \\ 1 & 0 & -1 \\ 0 & 1 & 0 \end{pmatrix} \right)$
    \item[] $C_4 = \left( \{\frac{x_1+1}{x_0}, x_1, x_2\}, \begin{pmatrix} 0 & -1 & 0 \\ 1 & 0 & 1 \\ 0 & -1 & 0 \end{pmatrix} \right)$
    \item[] $C_5 = \left( \{x_0, \frac{x_0x_2+x_1+1}{x_1x_2}, \frac{x_1+1}{x_2}\}, \begin{pmatrix} 0 & -1 & 1 \\ 1 & 0 & -1 \\ -1 & 1 & 0 \end{pmatrix} \right)$
    \item[] $C_6 = \left( \{\frac{x_1+1}{x_0}, x_1, \frac{x_1+1}{x_2}\}, \begin{pmatrix} 0 & -1 & 0 \\ 1 & 0 & 1 \\ 0 & -1 & 0 \end{pmatrix} \right)$
    \item[] $C_7 = \left( \{\frac{x_1+1}{x_0}, \frac{x_0x_2+x_1+1}{x_0x_1}, x_2\}, \begin{pmatrix} 0 & 1 & -1 \\ -1 & 0 & 1 \\ 1 & -1 & 0 \end{pmatrix} \right)$
    \item[] $C_8 = \left( \{x_0, \frac{x_0x_2+1}{x_1}, \frac{x_0x_2+x_1+1}{x_1x_2}\}, \begin{pmatrix} 0 & -1 & 0 \\ 1 & 0 & 1 \\ 0 & -1 & 0 \end{pmatrix} \right)$
    \item[] $C_9 = \left( \{\frac{x_0x_2+x_1+1}{x_0x_1}, \frac{x_0x_2+1}{x_1}, x_2\}, \begin{pmatrix} 0 & 1 & 0 \\ -1 & 0 & 1 \\ 0 & -1 & 0 \end{pmatrix} \right)$
    \item[] $C_{10} = \left( \{\frac{x_1+1}{x_0}, \frac{x_1^2+x_0x_2+2x_1+1}{x_0x_1x_2}, \frac{x_1+1}{x_2}\}, \begin{pmatrix} 0 & 1 & 0 \\ -1 & 0 & -1 \\ 0 & 1 & 0 \end{pmatrix} \right)$
    \item[] $C_{11} = \left( \{\frac{x_1+1}{x_0}, \frac{x_0x_2+x_1+1}{x_0x_1}, \frac{x_1^2+x_0x_2+2x_1+1}{x_0x_1x_2}\}, \begin{pmatrix} 0 & 0 & 1 \\ 0 & 0 & -1 \\ -1 & 1 & 0 \end{pmatrix} \right)$
    \item[] $C_{12} = \left( \{\frac{x_1^2+x_0x_2+2x_1+1}{x_0x_1x_2}, \frac{x_0x_2+x_1+1}{x_1x_2}, \frac{x_1+1}{x_2}\}, \begin{pmatrix} 0 & 1 & -1 \\ -1 & 0 & 0 \\ 1 & 0 & 0 \end{pmatrix} \right)$
    \item[] $C_{13} = \left( \{\frac{x_0x_2+x_1+1}{x_0x_1}, \frac{x_0x_2+1}{x_1}, \frac{x_0x_2+x_1+1}{x_1x_2}\}, \begin{pmatrix} 0 & 1 & 0 \\ -1 & 0 & -1 \\ 0 & 1 & 0 \end{pmatrix} \right)$
    \item[] $C_{14} = \left( \{\frac{x_0x_2+x_1+1}{x_1x_2}, \frac{x_0x_2+x_1+1}{x_0x_1}, \frac{x_1^2+x_0x_2+2x_1+1}{x_0x_1x_2}\}, \begin{pmatrix} 0 & 0 & -1 \\ 0 & 0 & 1 \\ 1 & -1 & 0 \end{pmatrix} \right)$
\end{itemize}

The seed $C_2' = \left( \{ \frac{x_1+1}{x_2},x_0, x_1\}, \begin{pmatrix} 0 & 0& -1 \\ 0 & 0 & 1 \\ 1& -1 & 0 \end{pmatrix} \right)$ is a mutation equivalent to $C_2$. In other words, it belongs to the same mutation class and  also $C_2'=\mu_2(\mu_1(C_8))$ and  $C_2'=\mu_1(\mu_0\mu_1\mu_2(\mu_0(\mu_1\mu_2(C_4)))))))$. Thus we can observe that in a mutation class a seed can be the result of applying sequences of mutations of different length and order.

Now, if the Dynkin diagram $\Gamma_r$ associated with the cluster algebra $A_r$ is known, the probability of choosing the mutation class where the initial seed of the cipher is found is $\frac{1}{N_C}$ and the probability of finding the initial seed in the mutation class is $\frac{1}{r!}$, remembering that the order of the cluster variables is important in the encryption and decryption of the message.

Therefore, the probability of finding the initial seed, if the Dynkin diagram $\Gamma_r$ is known, is $\frac{1}{N_C}\cdot \frac{1}{r!}=\frac{1}{N_c\cdot r!}$. We have the following
\begin{itemize}
\item If $\Gamma_r=A_r$, the probability is   $\frac{r(r+2)}{(2r+2)!}$ which tends to zero  as $r\rightarrow \infty$.
\item If $\Gamma_r=B_r$ or $\Gamma_r=C_r$, the probability is   $\frac{r!}{(2r)!}$. In both cases this probability tends to zero  as $r\rightarrow \infty$.
\item If  $\Gamma_r=D_r$, the probability is   $\frac{(r-1)!}{(3r-2)(2r-2)!}$ which tends to zero  as $r\rightarrow \infty$.

\end{itemize}

\begin{table}[h]
\centering
\begin{tabular}{|c|c|}
\hline
$\Gamma_r$ & $\frac{1}{N_c \cdot r}$ \\ \hline
$E_6$      & 1.66                     \\ \hline
$E_7$      & 4.76                     \\ \hline
$E_8$      & 9.88                     \\ \hline
$F_4$      & 3.9                      \\ \hline
$G_2$      & 0.0625                   \\ \hline
\end{tabular}
\end{table}

Therefore, it is recommended to use cluster algebras associated with Dynkin diagrams of type $A_r$, $B_r$, $C_r$ and $D_r$.


\begin{thebibliography}{ZZ}

\bibitem{grafvaluada}
V. Dlab, and C.M. Ringel. Indecomposable representations of graphs and algebras. {\itshape Mem. Amer. Math. Soc.} 6(1976), no. 173.

\bibitem{teor1}
G. Dupont An approach to non-simply laced cluster algebras. {\itshape J. Algebra320}(2008), no. 4, 1626-1661.

\bibitem{fomin1}
S. Fomin and A. Zelevinsky. Cluster algebras. I. Foundations. {\itshape J. Amer. Math. Soc.} 15 (2002), no. 2, 497-529.

\bibitem{fominfinito}
S. Fomin and A. Zelevinsky. Cluster algebras. II. Finite type classification. {\itshape Invent. Math.} 154 (2003), no. 1, 63-121.

\bibitem{algebralie}
K. Erdmann y M.J. Wildon. Introduction to Lie Algebras. Springer Undergraduate Mathematics Series. 2006.

\bibitem{Sage}
G. Musiker y C.Stump. A compendium on the cluster algebra and quiver package in Sage. {\itshape Sém. Lothar. Combin.}65 (2010/12), Art. B65d, 67pp.


\bibitem{Marsh}
R. J. Marsh.  Lecture Notes on Cluster Algebras {\itshape  European Mathematical Society}65 (2014), Art. B65d, 67pp.


\bibitem{teor2}
B. Zhu. Preprojective cluster variables of acyclic cluster algebras. {\itshape Comm. Algebra 35}(2007), no. 9, 2857-2871. 







\bibitem{FG2006}
V. Fock and A. Goncharov,
\newblock Moduli spaces of local systems and higher Teichm\"uller theory,
\newblock \emph{Publications Math\'ematiques de l'IH\'ES}, 103, 1--211, 2006.

\bibitem{Chicherin2020}
D. Chicherin, J. M. Henn, and G. Papathanasiou,
\newblock Cluster algebras for Feynman integrals,
\newblock arXiv:2012.12285, 2020. 

\bibitem{GekhtmanIzosimov2024}
M. Gekhtman and A. Izosimov,
\newblock Integrable systems and cluster algebras,
\newblock arXiv:2403.07287, 2024.

\bibitem{BrownYakimov2024}
K. Brown and M. Yakimov,
\newblock Poisson geometry and Azumaya loci of cluster algebras,
\newblock \emph{Advances in Mathematics}, 453, 109822, 2024. :contentReference[oaicite:0]{index=0}

\bibitem{Oberwolfach2024}
K. Baur, B. Marsh, R. Schiffler, and S. Schroll,
\newblock Cluster Algebras and Its Applications,
\newblock \emph{Oberwolfach Reports}, 21(1), 69--136, 2024. :contentReference[oaicite:1]{index=1}

\end{thebibliography}
\end{document}